\documentclass{amsart}

\usepackage{ amsmath,amssymb,amsbsy,amsfonts, latexsym,amsopn,amstext, amsxtra,euscript,amscd, hyperref}
\usepackage{amssymb,amsmath,amscd, amsfonts}
\usepackage{amsmath,amsbsy,amsfonts,amsopn,amstext}
\usepackage{euscript,amssymb,amsbsy,amsfonts,amsthm,latexsym,amsopn,amstext, amsxtra,euscript,amscd}

\usepackage{times,mathptmx}

\usepackage{amsmath}
\usepackage{amsthm}

\usepackage{xy}
\input xy
\xyoption{all}

\newtheorem{thm}{Theorem}[section]
\newtheorem{lem}[thm]{Lemma}
\newtheorem{cor}[thm]{Corollary}

\theoremstyle{definition}
\newtheorem{defn}[thm]{Definition}

\newtheorem{rem}[thm]{Remark}

\title{Invariant of Binary Forms}

\author{Vishwanath Krishnamoorthy}

\address{1300, Escorial Place, \# 207 Palm Beach Gardens, FL, 33410}

\email{vish\_w\_a@yahoo.com}

\author{Tanush Shaska}

\address{Department of Mathematics, University of Idaho, Moscow, ID, 83843.}

\email{tshaska@uidaho.edu}

\author{Helmut V\"olklein}
\address{Department of Mathematics, University of Florida, Gainesville, FL, 32611.}

\email{helmut@math.ufl.edu}

\begin{document}

\def\Z{\mathbb Z}
\def\cB{\mathcal B}
\def\cA{\mathcal A}
\def\C{\mathbb C}
\def\w{\widetilde}

\maketitle

\begin{abstract}
Basic invariants of binary forms over $\mathbb C$ up to degree 6 (and lower degrees) were constructed by Clebsch
and Bolza in the 19-th century using complicated symbolic calculations. Igusa extended this to algebraically
closed fields of any characteristic using difficult techniques of algebraic geometry. In this paper a simple proof
is supplied that works in characteristic $p > 5$ and uses some concepts of invariant theory developed by Hilbert
(in characteristic 0) and Mumford, Haboush et al. in positive characteristic. Further the analogue for  pairs of
binary cubics is also treated.
\end{abstract}

\section{Introduction}

Let $k$ be  an algebraically closed field of  characteristic not equal to  2.   A binary  form  of degree  $d$  is
a homogeneous  polynomial $f(X,Y)$ of  degree $d$ in two  variables over $k$.  Let  $V_d$ be the $k$- vector space
of binary  forms of degree $d$.  The group $GL_2(k)$ of  invertible  $2  \times 2$  matrices  over  $k$  acts on
$V_d$  by coordinate  change. Many  problems  in algebra  involve properties  of binary forms  which are invariant
under these  coordinate changes.  In particular, any  genus 2 curve over  $k$ has a  projective equation of the
form  $Z^2Y^4 = f(X,Y)$,  where $f$ is  a binary sextic  (= binary form  of degree  6)  of  non-zero discriminant.
Two  such curves  are isomorphic  if and  only if  the corresponding  sextics  are conjugate under $GL_2(k)$.
Therefore the moduli  space $\mathcal M_2$ of genus 2 curves  is the affine  variety whose  coordinate ring  is
the  ring of $GL_2(k)$-invariants in the coordinate ring  of the set of elements of $V_6$ with non-zero
discriminant.

Generators for this  and similar invariant rings in  lower degree were constructed by  Clebsch, Bolza  and others
in  the last  century using complicated calculations.  For the case of sextics,  Igusa \cite{Ig} extended  this to
algebraically closed  fields of  any characteristic using  difficult techniques  of algebraic  geometry. Igusa's
paper is very difficult to read and has some proofs only sketched. It is mostly the case of characteristic 2 which
complicates his paper.

Hilbert  \cite{Hi}  developed some  general, purely  algebraic tools (see Theorem 1 and Theorem 2 below) in
invariant theory. Combined with the linear reductivity of  $GL_2(k)$ in characteristic 0, this permits a  more
conceptual proof  of the  results of  Clebsch [${\bf  2}$] and Bolza  \cite{Bo}. After  Igusa's  paper appeared,
the concept  of geometric reductivity  was developed by Mumford \cite{Mu1}, Haboush \cite{Ha}  and others. In
particular it was proved  that reductive algebraic   groups    in   any   characteristic    are geometrically
reductive. This allows  application  of  Hilbert's  methods  in  any characteristic. For example, Hilbert's
finiteness theorem (see Theorem 1  below)  was  extended   to  any  characteristic  by  Nagata  \cite{Na}. Here we
give a proof of the Clebsch-Bolza-Igusa  result along those lines. The proof is  elementary in characteristic 0,
and extends to characteristic  $p >  5$  by  quoting  the respective  results  on geometric reductivity. This is
contained in sections 2 and 3.

In section 4  we treat the analogue for invariants  of pairs of binary cubics. To  our knowledge this has  not
been worked out  before.

\section{Invariants of Binary Forms}

In this chapter we define the action of $ GL_2(k)$ on binary forms and discuss the basic notions of their
invariants. Throughout this chapter $k$ denotes an algebraically closed field.

\subsection{Action of $GL_2(k)$ on binary forms.}

Let $k\, [X, Y]$  be the  polynomial ring in  two variables and  let $V_d$ denote  the  $d+1$-dimensional subspace
of $k\, [X, Y]$  consisting  of homogeneous polynomials.
\begin{equation}\label{eq1}
f(X,Y) = a_0X^d + a_1X^{d-1}Y +  \dots  + a_dY^d
\end{equation}
of  degree $d$. Elements  in $V_d$  are called  {\it binary  forms} of degree $d$.

We let $GL_2(k)$ act as a group of automorphisms on $ k\, [X,Y] $ as follows: if
$$ g = \begin{pmatrix} a & b \\ c & d \end{pmatrix}
\in GL_2(k) $$ then
\begin{equation}\label{eq2}
\begin{split}
g (X) = aX + bY \\
g (Y) = cX + dY
\end{split}
\end{equation}
This action of $GL_2(k)$  leaves $V_d$ invariant and acts irreducibly on $V_d$.

\begin{rem}\label{rem1}
It is well  known that $SL_2(k)$ leaves a bilinear  form (unique up to scalar multiples) on $V_d$ invariant. This
form is symmetric if $d$ is even and skew symmetric if $d$ is odd.
\end{rem}

Let $A_0$, $A_1$,  \dots ,   $A_d$ be coordinate  functions on $V_d$. Then the coordinate  ring of $V_d$ can be
identified with $ k\, [A_0  ,    \dots  ,   A_d] $. For $I \in k\, [A_0,  \dots  ,   A_d]$ and $g \in GL_2(k)$,
define $I^g \in k\, [A_0,  \dots  ,   A_d]$ as follows
\begin{equation}\label{eq3}
{I^g}\, (f) = I \, ( g ( f ) )
\end{equation}
for all $f \in V_d$. Then  $I^{gh} = (I^{g})^{h}$ and Eq.~\eqref{eq3} defines an action of $GL_2(k)$ on $k\, [A_0,
\dots  ,   A_d ]$.

\begin{defn}
Let  $\mathcal R_d$  be the  ring of  $SL_2(k)$ invariants  in $k\, [A_0,  \dots  ,   A_d]$, i.e., the ring of all
$I \in k \, [A_0,  \dots  ,   _d]$ with $I^g = I$ for all $g \in SL_2(k)$.
\end{defn}

Note that if $I$ is an invariant, so are all its homogeneous components. So $\mathcal  R_d$ is  graded by  the
usual degree  function on  $k\, [A_0,  \dots  ,   A_d]$.

Since $k$ is algebraically closed, the binary form $f(X,Y)$ in  Eq.~\eqref{eq1} can be factored as
\begin{equation} \label{eq4}
f(X,Y)  = (y_1  X  -  x_1 Y) \cdots (y_d  X  - x_d  Y) = \displaystyle \prod_{1 \leq  i \leq  d} \det
\left(\begin{pmatrix} X & x_{i} \\ Y & y_i
\end{pmatrix} \right)
\end{equation}
The points  with homogeneous coordinates $(x_i, y_i)  \in \mathbb P^1$ are  called the  roots  of the  binary form
\eqref{eq1}.  Thus  for $g  \in GL_2(k)$ we have
$$g\left (f(X,Y) \right )  = ( \det(g))^{d}  (y_1^{'}  X -  x_1^{'}  Y) \cdots (y_d^{'} X  - x_d^{'} Y),$$
where
\begin{equation}
\begin{pmatrix}   x_i^{'}  \\  y_i^{'}  \end{pmatrix} = g^{-1}  \begin{pmatrix} x_i\\ y_i \end{pmatrix}.
\end{equation}

\subsection{The Null Cone of \, $V_d$}

\begin{defn}
The  null cone  $N_d$ of  $V_d$  is the  zero set  of all  homogeneous elements in $\mathcal R_d$ of positive
degree
\end{defn}

\begin{lem}\label{lem1}
Let $char(k)  = 0$  and $\Omega_s$  be the subspace  of $k\, [A_0,   \dots  ,   A_d]$ consisting of homogeneous
elements of degree $s$. Then there is a $k$-linear  map $R :  k\, [A_0,  \dots  ,    A_d] \to \mathcal R_d$  with
the following properties:

\smallskip

(a) $R(\Omega_s) \subseteq  \Omega_s$ for all $s$

(b) $R(I) = I$ for all $I \in \mathcal R_d$

(c) $R(g(f)) = R(f)$ for all $f \in k\, [A_0,  \dots  ,   A_d]$
\end{lem}

\begin{proof} $\Omega_s$ is a  polynomial module of degree $s$  for $SL_2(k)$. Since $SL_2(k)$  is linearly  reductive
in $char(k)  = 0$,  there exists  a $SL_2(k)$-invariant  subspace  $\Lambda_s$  of  $\Omega_s$  such  that
$\Omega_s = (\Omega_s \cap  \mathcal R_d) \bigoplus \Lambda_s$. Define $R : k\, [A_0,   \dots  ,   A_d] \to
\mathcal R_d $ as $R(\Lambda_s)  = 0$ and $R_{|\Omega_s \cap \mathcal R_d} = id$. Then $R$ is $k$-linear and the
rest of the proof is clear from the definition of $R$.

\end{proof}

\noindent The map $R$ is called the {\bf Reynold's operator}.

\begin{lem}\label{lem2}
Suppose $char(k) = 0$.  Then every maximal ideal in $ \mathcal R_d$ is contained in a maximal ideal of $k\, [A_0,
\dots  ,   A_d]$.
\end{lem}

\begin{proof}
If $\mathcal I$ is a maximal ideal in $ \mathcal R_d $ which generates the unit ideal of
 $ k\, [A_0, $ $\dots  ,   A_d]$, then there exist $m_1,   \dots ,   m_t \in \mathcal I$ and $f_1$, $f_2$,  \dots
, $f_t \in k\, [A_0,  \dots  ,   A_d]$ such that
$$1 = m_1 f_1 +  \dots  + m_t f_t$$
Applying the Reynold's operator to the above equation we get
$$1 = m_1 \, R(f_1) +  \dots  + m_t \, R(f_t)$$
But  $R(f_i)  \in \mathcal  R_d$  for all  $i$.  This  implies $1  \in \mathcal I$, a contradiction.

\end{proof}

\begin{thm}\label{thm1} (\textbf{Hilbert's Finiteness  Theorem})  Suppose $char(k)  = 0$.  Then $\mathcal R_d$ is finitely generated over $k$.
\end{thm}

\begin{proof}
Let $\mathcal I_0$ be the ideal in $k\, [A_0,  \dots  ,   A_d]$ generated by all homogeneous invariants of
positive degree. Because $k\, [A_0,  \dots  ,  A_d]$ is Noetherian,  there exist finitely many  homogeneous
elements $J_1,
 \dots  ,   J_r$  in $\mathcal  R_d$   such  that  $\mathcal  I_0  =  (J_1,  \dots  ,  J_r)$.
We prove $\mathcal R_d =  k\, [J_1,  \dots  ,   J_r]$. Let $J \in \mathcal R_d$  be homogeneous  of degree $d$. We
prove $J  \in k\, [J_1, \dots ,   J_r]$ using induction on $d$.  If $d = 0$, then $J \in k \subset k\, [J_1, \dots
,   J_r]$. If $d > 0$, then
\begin{equation}\label{eq_6}
J = f_1 \, J_1 +  \dots  + f_r \, J_r
\end{equation}
with $f_i  \in k\, [A_0,  \dots  ,    A_d]$ homogeneous and $deg(f_i)  < d$ for all $i$. Applying the Reynold's
operator to Eq.~\eqref{eq_6} we have
$$J = R(f_1) J_1 +  \dots  + R(f_r) J_r$$
then by  Lemma 1 $R(f_i)$ is  a homogeneous element  in $\mathcal R_d$ with $deg(R(f_i))  < d $  for all $i$  and
hence by induction  we have $R(f_i) \in k\, [J_1,  \dots   ,   J_r]$ for all $i$. Thus $J  \in k\, [J_1,  \dots  ,
J_r]$.

\end{proof}

If $k$ is of arbitrary characteristic, then $SL_2(k)$ is geometrically reductive,  which is a  weakening of linear
reductivity; see Haboush \cite{Ha}. It suffices to prove Hilbert's finiteness theorem in any characteristic; see
Nagata \cite{Na}. The following theorem is also due to Hilbert.

\begin{thm}\label{thm2}
Let $I_1$,  $I_2$,  \dots ,    $I_s$ be homogeneous  elements in  $ \mathcal R_d$ whose common zero set equals the
null cone $\mathcal N_d$. Then $ \mathcal R_d$  is finitely  generated as a  module over $k\, [I_1,   \dots  ,
I_s]$.
\end{thm}

\begin{proof}
(i) $char(k)  = 0$: By Theorem \ref{thm1}  we have $\mathcal R_d  = k\, [J_1, J_2,  \dots  ,    J_r]$ for  some
homogeneous invariants  $J_1$,  \dots   ,  $J_r$. Let $\mathcal I_0$  be the maximal ideal  in $ \mathcal  R_d$
generated by all homogeneous elements  in $ \mathcal R_d$ of  positive degree. Then the theorem follows if $I_1$,
\dots , $I_s$ generate an ideal $\mathcal I$ in $ \mathcal R_d$ with $rad(\mathcal I)=\mathcal I_0$. For if this
is the case, we have an integer $q$ such that
\begin{equation}\label{eq5}
 J_i ^{q} \in \mathcal I, \quad  \textit{    for all   } i
\end{equation}
Set $S:= \{ J_1 ^{i_1} J_2^{i_2} \dots J_r^{i_r} \,  | \, 0 \leq i_1,  \dots  ,   i_r <  q \}$.  Let $\mathcal  M$
be the $k\, [I_1,  \dots   I_s]$-submodule in $\mathcal R_d$  generated by $S$.  We prove $\mathcal R_d  =
\mathcal M$. Let $J \in \mathcal  R_d$  be homogeneous.  Then $J  = J^{'}  + J^{''}$  where  $J^{'} \in  \mathcal
M$, $\, \, J^{''}$ is a $k$-linear combination of  $J_1^{i_1} J_2 ^{i_2}   \dots  J_r ^{i_r}$ with  at least one
$i_{\nu} \geq q$ and $deg(J) = deg(J^{'})  = deg(J^{''})$. Hence Eq.~\eqref{eq5} implies $J^{''} \in \mathcal I$
and so we have
$$J^{''} = f_1 \, I_1 + \cdots  + f_s \, I_s$$
where $f_i \in \mathcal R_d$ for all $i$. Then $deg(f_i) < deg(J^{''}) = deg(J)$ for all $i$. Now by induction on
degree of $J$ we may assume $f_i \in \mathcal M$ for all $i$. This implies $J^{''} \in \mathcal M$ and  hence  $J
\in  \mathcal  M$.  Therefore  $\mathcal M  =  \mathcal R_d$.  So  it  only  remains  to prove  $rad(\mathcal  I)
=  \mathcal I_0$. This  follows from  Hilbert's Nullstellensatz and  the following claim.

\medskip

\noindent  \textbf{Claim:}  $\mathcal I_0$ is the only maximal ideal containing $I_1,  \dots  ,   I_s$.

\medskip

Suppose $\mathcal I_1 $ is  a maximal ideal  in $ \mathcal  R_d$ with $I_1,  \dots   ,   I_s \in  \mathcal I_1$.
Then from Lemma  2 we know there exists a  maximal ideal  $\mathcal J$  of $k\, [A_0,  \dots   ,   A_d]$  with $
\mathcal I_1 \subset \mathcal J$.  The point in $V_d$ corresponding to $\mathcal  J$ lies  on the  null  cone
$\mathcal N_d$ because  $I_1,  \dots   ,   I_s$ vanish  on  this point.  Therefore  $\mathcal I_0  \subset
\mathcal J$,  by definition of  $\mathcal N_d$. Therefore  $\mathcal J \cap \mathcal R_d$ contains both the
maximal ideals $\mathcal I_1$ and $\mathcal I_0$. Hence,   $\mathcal I_1 = \mathcal J  \cap \mathcal R_d =
\mathcal I_0$.

(ii)  $char(k) = p$: The same proof works if Lemma 2 holds. Geometrically  this means the morphism  $\pi : V_d \to
V_d$ // $SL_2(k)$ corresponding to the  inclusion $\mathcal R_d \subset k\, [A_0,  \dots  ,   A_d]$ is surjective.
Here $ V_d$ // $SL_2(k)$ denotes the affine variety corresponding  to the  ring $\mathcal R_d$  and is  called the
{\it categorical quotient}. $\pi$ is surjective  because $SL_2(k)$ is geometrically  reductive. The proof  is by
reduction modulo $p$, see Geyer \cite{Ge}.

\end{proof}

\section{Projective Invariance of Binary Sextics.}
Throughout this section $char(k) \neq 2,3,5$
%
\subsection{Construction of invariants and characterization of multiplicities of the roots.}
We let
\begin{equation}
\begin{split}
f(X,Y)       &=      a_0X^6+a_1X^5Y+ \dots +a_6Y^6       \\
&      =  (y_1X-x_1Y)(y_2X-x_2Y) \dots (y_6X-x_6Y)
\end{split}
\end{equation}

\noindent be an element in $ V_6$.  Set
$$D_{ij}:=   \begin{pmatrix} x_i & x_j \\  y_i & y_j \end{pmatrix}.$$  For $g  \in SL_2(k)$,  we have
$$g(f) =  (y_1^{'} X  - x_1^{'} Y)   \dots  (y_6^{'}  X - x_6^{'}  Y), \quad  \textit{ with } \quad
 \begin{pmatrix}   x_i^{'}  \\ y_i^{'}  \end{pmatrix}  = g^{-1} \, \begin{pmatrix} x_i\\ y_i \end{pmatrix}.$$ Clearly
$D_{ij}$ is  invariant under this action of $SL_2(k)$ on $\mathbb  P^{1}$. Let $\{i, j,  k, l, m, n  \} =\{ 1, 2,
3,  4, 5, 6 \}$. Treating $a_i$ as  variables, we construct the following elements in $\mathcal R_6$ (proof
follows).

\begin{equation}
\begin{split}
I_{10} & =   \prod_{ i <  j } D_{ij}^2 \\
I_2  & =  \displaystyle \sum_{\substack {i<j,k<l,m<n}}  D_{ij}^2D_{kl}^2D_{mn}^2 \\
I_4 &= (4I_{2}^2-B)\\
I_6 &= (8I_{2}^3-160I_2I_4-C) \\
\end{split}
\end{equation}
where
\begin{equation}
\begin{split}
B &=  \sum_{\substack {i<j,j<k,l<m,m<n}}
D_{ij}^2D_{jk}^2D_{ki}^2D_{lm}^2D_{mn}^2D_{nl}^2 \\
C  & = \sum_{
\substack {i<j,j<k,l<m,m<n \\
i<l^{'},j<m^{'},k<n^{'} \\
l^{'},m^{'},n^{'} \in \{l,m,n\}}} D_{ij}^2 D_{jk}^2 D_{ki}^2 D_{lm}^2 D_{mn}^2 D_{nl}^2
D_{{il}^{'}}^2 D_{{jm}^{'}}^2 D_{{kn}^{'}}^2 \\
\end{split}
\end{equation}
The  number of  summands  in  $B$ \ (resp.  $C$)  equals $\frac{  \begin{pmatrix}  6 \\ 3
\end{pmatrix} }{ 2 !} =  10 $ (resp. 60).

\begin{lem}\label{lem3}
$I_{2i}$ are homogeneous elements in $\mathcal R_6$ of degree $2i$,  for $i$ = 1,2,3,5.
\end{lem}

\begin{proof}
Each $I_{2i}$  can be written  as
$$(y_1 \dots y_6)^{2i} \cdot {\w I}_{2i} (\frac{x_1} {y_1},  \dots  ,   \frac{x_6} {y_6})$$
with ${\w I}_{2i}$ a symmetric  polynomial  in  $\frac{x_1}{y_1},  \frac{x_1}{y_2},   \dots   ,   \frac{x_6}{y_6}$
for $i$ = 1, 2, 3, 5.  Therefore  by the  fundamental theorem of elementary symmetric functions  we have
$$I_{2i} = a_0 ^{2i} \cdot f_i(\frac{a_1}{a_0},   \dots   ,     \frac{a_6}{a_0}),$$
where  $f_i$  is  a polynomial in 6 variables and hence $I_{2i}$ is a rational function in $a_0$,   \dots  $a_6$
with denominator  a power  of $a_0$.  Switching the roles $X$ and $Y$  we also  see that  the denominator  is a
power of $a_6$.  Thus $I_{2i}  \in k\, [a_0,   \dots   ,   a_6]$. Clearly $I_{2i}$  are $SL_2(k)$-invariants  and
hence lie in  $\mathcal  R_6$.  Further, replacing  $f$ by $cf$  with $  c \in  k^{*}$, multiplies  $I_{2i}$ by
$c^{2i}$. Hence, $I_{2i}$ are homogeneous of degree $2i$.
\end{proof}

Note  that $I_2$ is  the $SL_2(k)$-invariant  quadratic form  on $V_6$ (see  {\it  Remark  \ref{rem1}}) and
$I_{10} $ is the discriminant  of  the sextic.  $I_{10}$   vanishes  if  and   only  if  two  of   the  roots
coincide. Also note that if for a sextic all its roots are equal, then all the basic invariants vanish. These
basic invariants when evaluated on a sextic $f(X,Y)  = a_0 X^6 + a_1 X^5 Y +  \dots   a_6 Y^6$ with a root at
$(1,0)$, i.e., with $a_0 = 0$, take the following form.

\begin{small}
\begin{equation}
\begin{split}
I_2 = &-20a_1a_5+8a_2a_4-3a_3^2 \\
I_4 = & -24000a_1^2a_4a_6+10000a_1^2a_5^2+14400a_1a_3a_2a_6-1800a_1a_3^2a_5-3200a_1a_4a_2a_5\\
& +960a_1a_3a_4^2-3840a_2^3a_6+960a_2^2a_3a_5+256a_2^2a_4^2-432a_2a_4a_3^2+81a_3^4\\
I_6 = & 100a_1a_3^4a_5-40a_1a_3^3a_4^2+6250a_1^3a_3a_6^2-160a_2^4a_4a_6+60a_2^3a_3^2a_6\\
& -40a_2^2a_3^3a_5-8a_2^2a_3^2a_4^2-2500a_2^2a_1^2a_6^2+8a_2a_3^4a_4-2500a_1^2a_3a_6a_2a_5\\
& -100a_2^4a_5^2-24a_2^3a_4^3-350a_1a_3^2a_2a_4a_5+300a_1a_3a_2^2a_4a_6+1000a_2^3a_1a_6a_5\\
& -100a_1^2a_4^4-a_3^6+250a_1^2a_3^2a_6a_4+250a_1^2a_4^2a_3a_5-100a_1a_4^2a_2^2a_5\\
& +250a_1a_3a_2^2a_5^2+140a_2^3a_4a_3a_5-150a_1a_3^3a_2a_6+140a_1a_3a_2a_4^3 \\
\end{split}
\end{equation}
\end{small}

\begin{lem}\label{lem4}
A sextic has a root of multiplicity exactly three if and only if the basic invariants take the form
\begin{equation}\label{eq6}
I_2 = 3r^{2}, \quad I_4 =  81r^{4}, \quad I_6 = r^{6},  \quad I_{10} = 0.
\end{equation}
for some $ r \neq 0 $.
\end{lem}

\begin{proof}
Let $f(X,Y) =  a_0 X^6 + a_1 X^5Y  +  \dots  + a_6 Y^6$ be  a sextic with  triple root. Let the triple root be at
$(1,0)$. Then $a_0 = a_1 =  a_2  = 0$. Set  $ a_3 = r  $. Then $I_{2i}$ for $  i $ = 1, 2,  3 take the form
mentioned in the lemma. Conversely assume Eq.~\eqref{eq6}. Since $I_{10} = 0 $, the sextic has a multiple  root.
Since $I_6 \neq 0$, there is at least one more  root. We assume the multiple root is  at $(1,0) $ and other root
is $ ( 0,1) $. Then the sextic takes the form
$$ a_{2}X^4 Y^2 + a_{3}X^3 Y^3 + a_{4}X^2 Y^4 + a_{5}X Y^5 $$
and  Eq.~\eqref{eq6} becomes
\begin{small}
\begin{equation}\label{eq7}
\begin{split}
-8a_2a_4+3a_3^2 = & \,  3r^2\\
960a_2^2a_3a_5+256a_2^2a_4^2-432a_2a_4a_3^2+81a_3^4 = & \, 81r^4\\
40a_2^2a_3^3a_5+8a_2^2a_3^2a_4^2-8a_2a_3^4a_4+24a_2^3a_4^3+100a_2^4a_5^2-140a_2^3a_4a_3a_5+a_3^6 = & \,  r^6
\end{split}
\end{equation}
\end{small}
Now eliminating $ a_4 $ from Eq.~\eqref{eq7}, we have,
$$ 2^{6} a_{2}^2 a_3{a_5} = 3(a_3^{2} - r^2)^2 \quad  \textit{  and } \quad  2^{9}{a_2}^4{a_5}^2 = (a_3^{2} - r^2)^3.$$
Eliminating $ a_2$ and $ a_5 $ from these equations we get
$$ (a_3^{2} - r^2)^3({a_3}^2 - (3r)^2) = 0.$$
If ${a_3}^2 = r^2 $, then $ {a_2} \, a_4 = a_2 \, a_5 = 0 $. In this case either $(0,1)$ or $ (1,0) $ is a triple
root. On the other hand if we have $ {a_3}^2 = (3r)^2 $, then $a_2\, a_4 = 3r^2$ and $ {a_2}^2\, a_5 = r^3 $ or $
-r^3$. Hence, either $(r{a_2}^{-1}, 1)$ or $(-r{a_2}^{-1}, 1)$ is a triple root.

\end{proof}

\begin{lem}\label{lem5}
A sextic has a root of multiplicity at least four if and only if the basic invariants vanish simultaneously.
\end{lem}

\begin{proof}
Suppose $(1, 0)$ is a root of multiplicity 4. Then $a_1 = a_2 = a_3 = 0 $. Therefore $I_2 = I_4 = I_6 = I_{10} =
0$. For the converse, since $I_{10} = 0$, there is a multiple root. If there is no root other than the multiple
root, we are done. Otherwise, let the multiple root be at (1,0) and the other root be at (0, 1). Then as in the
previous lemma, the sextic becomes
$$ a_{2}X^4 Y^2 + a_{3}X^3 Y^3 + a_{4}X^2 Y^4 + a_{5}X Y^5 $$
Now $ I_2 = 0 $ implies $ a_2\, a_4 = 2^{-3} \cdot 3 \cdot  {a_3}^2 $ and hence $ I_4 = 0 $ implies
$$ {a_2}^2\, a_3\, a_5 = 2^{-6}\cdot 3 \cdot {a_3}^4 .$$
Using these two equations in $ I_6 = 0 $ we find $ a_2\, a_3 = 0 $. Let $a_2 \neq 0$. This implies $a_3 =  a_4 =
a_5 = 0 $ and the sextic has a  root of multiplicity four at $(0,1)$. If $ a_2 = 0$, then $ I_2 = 0 $ implies $a_3
= 0 $ and therefore the sextic has a root of multiplicity four at $(1,0)$.
\end{proof}

\subsection{The Null Cone of $V_6$ and Algebraic Dependencies}
\begin{lem}\label{lem6}
$\mathcal R_6$ is finitely generated as a module over $k\, [I_2, I_4, I_6, I_{10}]$.
\end{lem}

\begin{proof}
By Theorem \ref{thm2} we only have to prove $\mathcal N_6 = V(I_2, I_4, I_6 ,  I_{10}) $. For $\lambda \in k^{*}$,
set $ g(\lambda):= \left(\begin{pmatrix} \lambda ^{-1} &0\\ 0 &\lambda
\end{pmatrix} \right)$. Suppose $I_2$, $I_4$, $I_6$ and $I_{10}$ vanish on a sextic $f \in V_6$.
Then  we know from Lemma \ref{lem5} that $f$ has a root of multiplicity at least 4. Let this multiple root be
$(1,0)$. Then $f$ is of the form
$$f(X,Y) = ({a_4}X^2 + {a_5}XY + {a_6}Y^2)Y^{4}.$$
If $I \in \mathcal R_6$ is homogeneous of degree $s > 0$, then
$$I(f^{g(\lambda)}) = \lambda ^{2s}I_s (a_4 X^2 Y^4 + a_5 \lambda ^2 XY^5  + a_6 \lambda ^4 Y^6).$$
Thus $I(f^{g(\lambda)})$ is a polynomial in $\lambda$ with no constant term. But since $I$ is an
$SL_2(k)$-invariant, we have $I(f^{g(\lambda)}) = I(f)$ for all $\lambda$. Thus $I(f) = 0$. This proves the null
cone $\mathcal N_6 = V(I_2, I_4, I_6, I_{10})$.
\end{proof}

\begin{rem}\label{rem2}
(a) Lemma \ref{lem6} implies $I_2$, $I_4$, $I_6$ and $I_{10}$ are algebraically independent over $k$ because
$\mathcal R_6$ is the coordinate ring of the four dimensional variety $V_6$ // $SL_2(k)$.

(b) The quotient of two homogeneous elements in $k\, [I_2, I_4, I_6, I_{10}]$ of same degree in $A_0$, $A_1$,
\dots ,  $A_6$ is a $GL_2(k)$-invariant. In particular the following elements are $GL_2(k)$-invariants.
\begin{equation*}
T_1 := \frac{I_4}{I_2 ^{2}}, \quad T_2 := \frac{I_6}{I_2 ^{3}}, \quad T_3 := \frac{I_{10}}{I_2 ^{5}}
\end{equation*}

(c) Assertion (a) implies $T_1$, $T_2$ and $T_3$ are algebraically independent over $k$. For if there exists an
equation
\begin{equation}\label{eq3.3}
 \sum  a_{e f g}{T_1^{e} T_2^{f} T_3^{g}} = 0.
\end{equation}
Multiplying Eq.~\eqref{eq3.3}  by $I_2 ^{h}$ gives
\begin{equation}\label{eq3.4}
\sum a_{efg} I_4 ^{e} I_6 ^{f} I_{10}^{g} I_2 ^{h-2e-3f-5g} = 0.
\end{equation}
For large $h$, Eq.~\eqref{eq3.4} is a nontrivial polynomial relation between $I_2$, $I_4$, $I_6$ and $I_{10}$.
This contradicts (a).
\end{rem}
Further define the following
\begin{small}
\begin{equation}
\begin{split}
U_1 & := \frac{I_2^{5}}{I_{10}}   = \frac{1}{T_3}, \quad    U_2:= \frac{I_2 ^{3}I_4}{I_{10}} = \frac{T_1}{T_3},
\quad U_3:= \frac{I_2 ^{2}I_6}{I_{10}}=\frac{T_2}{T_3}, \quad
U_4:= \frac{I_4 ^{5}}{I_{10} ^{2}} = \frac{T_1 ^{5}}{T_3 ^{2}} \\
U_5 & :=\frac{I_4 I_6}{I_{10}}=\frac{T_1 T_2 }{T_3}, \quad U_6:= \frac{I_6^{5}}{I_{10} ^{3}}=\frac{T_2 ^{5}}{T_3
^{3}}, \quad U_7:= \frac{I_2 I_4 ^{2}}{I_{10}}=\frac{T_1 ^{2}}{T_3}, \quad U_8:=  \frac{I_2 I_6 ^{3}}{I_{10}
^{2}}=\frac{T_2 ^{3}}{T_3 ^{2}}.\\
\end{split}
\end{equation}
\end{small}

\begin{rem}\label{rem3}
From the definitions of $U_1$, $U_2$ and $U_3$ it is clear that $k\, (U_1, U_2, U_3)$ $ = k\, (T_1, T_2, T_3)$.
Therefore $U_1$, $U_2$ and $U_3$ are also algebraically independent over $k$.
\end{rem}

\begin{lem}\label{lem7}
Let $a$, $b$, $c$ and $d$ be non-negative integers such that $a+2b+3c=5d$. Then,
$${\bf m} = \frac{I_2^{a}\cdot  I_4^{b} \cdot I_6^{c}} {I_{10}^{d}}  \in  k \, [ U_1, U_2,  \dots  ,   U_8 ] $$
\end{lem}

\begin{proof}
From first column in the above table we see that it is enough to prove the lemma for non-negative integers $a$,
$b$, $c$, $d$ $ < 5$. The proof is now by inspection.
\end{proof}

\begin{lem}\label{lem8}
$\mathcal R := k\, [U_1, U_2, U_3, U_4, U_5, U_6, U_7, U_8]$ is normal
\end{lem}

\begin{proof}
Suppose an element $J$ in the field of fractions of $\mathcal R$ is integral over $\mathcal R$. Then we have an
equation
\begin{equation}\label{eq8}
J^n + p_{n-1}(U_1,  \dots ,   U_8) J^{n-1} +  \dots  + p_0(U_1,  \dots  ,  U_8) = 0
\end{equation}
where $p_i$ is a polynomial in 8 variables over $k$. Let $e$ be a positive integer such that $I_{10}^{e}\,  p_i\in
k\, [I_2, I_4, I_6, I_{10}]$ for all $i$. Then multiplying Eq.~\eqref{eq8} by $I_{10}^{n\, e}$, we see that $
I_{10}^e \, J$ is integral over $k\, [I_2, I_4, I_6, I_{10}]$. By {\it Remark 2 (a) } we know that $k\, [I_2, I_4,
I_6, I_{10}]$ is a polynomial ring. Also the field of fractions of $\mathcal R$ is contained in $k\, (I_2, I_4,
I_6, I_{10})$. Therefore $I_{10}^e J \in k\, [I_2, I_4, I_6, I_{10}]$. Since $I_{10}^{e} \, J $ is a homogeneous
element of degree $10e$ in $k\, [A_0,  \dots  ,  A_6]$, $J$ is a $k$- linear combination of elements of the form
${\bf m}$ in Lemma \ref{lem7}. Therefore $J \in \mathcal R$. Hence the claim.

\end{proof}

\subsection{ The Field of Invariants of $GL_2(k)$ on $ k ( A_0,  \dots  ,   A_6 ) $}

Let $K$ denote the invariant field under the $GL_2(k)$ action on $k(A_0,  \dots  ,   A_6)$.

\begin{thm}\label{thm3}
The field $K$ of $GL_2(k)$ invariants in $ k(A_0,  \dots  ,   A_6)$ is a rational functional field, namely $K =
k(T_1, T_2, T_3) = k(U_1, U_2, U_3)$.
\end{thm}

\noindent  Remark \ref{rem3} implies we only have to show $K = k(T_1, T_2, T_3)$. The proof occupies the remainder
of this section.

\begin{rem}\label{rem4}
If $\frac{R}{S} \in K$ with $R$ and $S$ coprime polynomials, then $\frac{R}{S} = \frac{R^{g}}{S^{g}}$ for every $g
\in GL_2(k)$. Since $R$ and $S$ are coprime we have $R = {c_g} R^g$ and $S = {c_g} S^{g}$ with $c_g \in k^{*}$ for
every $g \in GL_2(k)$. Hence $R$ and $S$ are homogeneous of same degree. The map $g \mapsto c_g$ is a group
homomorphism $GL_2(k) \to k^{*}$. Since $SL_2(k)$ is a perfect group, it is in its kernel. Thus $R$, $S \in
\mathcal R_6$.
\end{rem}

We introduce the following notations.
\begin{equation}
\begin{split}
& \mathcal U ^{(6)}   := \{ (p_1, p_2,  \dots  ,  p_6) : p_i \in \mathbb P^{1}, p_i \neq p_j \hspace{0.03in}
\forall i, j
\}\\
& \cA   := \{ f \in V_6 : I_{10}(f) \neq  0 \} \\
& \mathcal C := \{ (0, 1, \infty, c_1, c_2, c_3) : c_i \in k  - \{0,1 \}, c_i \neq c_j \hspace{0.03in} \forall i,j
\} \subseteq \mathcal U^{(6)}\\
& \cB := \{ f= X Y ( X - Y ) f_3 :
 f_3 = X^3-{b_1}X^2 Y+{b_2}XY^2-b_3 Y^3 = \\
 &  \quad  (X- c_1 Y)(X- c_2 Y)(X- c_3 Y), (0, 1, \infty, c_1, c_2, c_3)
\in \mathcal C  \}\\
\end{split}
\end{equation}
Then we have $k(\cB) = k(B_1, B_2, B_3)$ where $B_i$ is the function mapping $X Y (X-Y)(X^3-{b_1}X^2 Y+{b_2}
XY^2-b_3 Y^3)$ to $b_i$. Similarly $k(\mathcal C) = k(C_1, C_2, C_3)$.

$S_6$ acts on $\mathcal U^{(6)}$ by $(p_1, p_2,  \dots  ,  p_6) {\buildrel \tau \over  \mapsto} (p_{\tau(1)},
\dots ,  p_{\tau(6)})$ and $GL_2(k)$ acts on $\mathcal U^{(6)}$ by $(p_1, p_2,  \dots  ,  p_6) {\buildrel g \over
\mapsto} (g^{-1}(p_1),  \dots  ,  g^{-1}(p_6) )$.  These actions commute. This induces an action of $S_6$ on
$\mathcal U^{(6)}$ / $PGL_2(k)$. Each $PGL_2(k)$ orbit meets $\mathcal C$ in precisely one point. Therefore $
\mathcal U^{(6)}$ / $PGL_2(k) \cong \mathcal C $ and we have an action of $S_6$ on $\mathcal C$ and hence on
$k(C_1, C_2, C_3)$. If $\tau_{ij}$ is the transposition $(i,j)$, the $S_6$ action on $k ( C_1, C_2, C_3 )$ is
explicitly given as follows.

\medskip

(a). $ ( C_1, C_2, C_3) {\buildrel \tau_{12} \over  \longmapsto} ( 1- C_1, 1- C_2, 1- C_3)$\\

(b). $ (C_1, C_2, C_3) {\buildrel \tau_{23} \over  \longmapsto} ( \frac{C_1}{C_1 -1}, \frac{C_2}{C_2 -1},
\frac{C_3}{C_3 - 1})$.\\

(c). $ (C_1, C_2, C_3) {\buildrel \tau_{34} \over  \longmapsto} ( 1- C_1, \frac{C_2 (1 - C_1)}{C_2 - C_1},
\frac{C_3 (1- C_1)}{C_3 - C_1} ) $.\\

(d). $ (C_1, C_2, C_3) {\buildrel \tau_{45} \over  \longmapsto} ( C_2, C_1, C_3)$.\\

(e). $ (C_1, C_2, C_3) {\buildrel \tau_{56} \over  \longmapsto} ( C_1, C_3, C_2) $.\\

\smallskip

\noindent Let $F$ denote the fixed field of $S_6$ action on $k(C_1, C_2, C_3)$. The natural map $\mathcal C \to
\cB$ given by $$(0, 1, \infty, c_1, c_2, c_3) \mapsto X Y( X-Y )( X-c_1 Y )( X-c_2 Y )( X-c_3 Y)$$ induces a
Galois extension $C ( C_1, C_2, C_3 )$ / $k ( B_1, B_2, B_3 )$ with Galois group $S_3 < S_6$, where $S_3$ is
embedded as the subgroup of $S_6$ permuting the letters 4, 5, 6 and fixing 1, 2, 3.
\begin{figure}[ht!]
$$
\xymatrix{
& & k(C_1, C_2, C_3) \ar@{-}[d]^{S_3} & &\\
& & k(B_1, B_2, B_3) \ar@{-}[d]^{120} & & \\
& & F } $$
\end{figure}
\begin{lem}\label{lem9}
The inclusion $ \cB \subset V_6$ induces an embedding  $$K \subseteq F \subset k(B_1, B_2, B_3).$$
\end{lem}

\begin{proof}
$\cB \subset \cA$ and every element in $\cA$ is $GL_2(k)$- conjugate to a unique element in $\cB$. Recall by Remark
\ref{rem4}, if $\frac{R}{S} \in K$ with $R$ and $S$ coprime polynomials, then $S = c_g{S^g}$ for all $g \in
GL_2(k)$. If $S$ vanishes on $\cB$, it also vanishes on $\cA$. But $\cA$ is open in $ k ^{6}$ and so  $S \equiv 0$
which is a contradiction. Therefore $S$ does not vanish on $\cB$ and hence the restriction map $K \to k(\cB)$ is
well defined. Thus we have $K \subset k(\cB) \subset k(\mathcal C)$. Let $I \in K$ and $\bar I $ its image in
$k(\mathcal C) = k(k_1, C_2, C_3)$. Denote $ p = (0 ,   1, \infty, c_1, c_2, c_3) \in \mathcal U ^{(6)} $ by $(p_1,
\dots  ,  p_6)$. For $\tau \in S_6$ we have
\begin{equation*}
\begin{split}
\bar I (p^{\tau}) & =  \bar I ( g(p_{\tau(1)}),  \dots  ,  g(p_{\tau(6)})) \\
 & = I ((X-g(p_{\tau(1)})Y)  \dots (X-g(p_{\tau(6)})Y))  \\
 &  = I( (X-p_{\tau(1)}Y)  \dots  (X-p_{\tau(6)}Y)) = \bar I(p)
\end{split}
\end{equation*}
for some $g \in GL_2(k)$ and so the lemma follows.
\end{proof}

Let us now see how the elements  $T_i$ of $K$ embed in $k(B_1, B_2, B_3)$. Evaluating $I_{2i}$ on sextics of the
form $XY(X-Y)(b_0X^3-b_1X^2 Y+b_2XY^2-b_3 Y^3)$ yields the following homogeneous polynomials $J_{2i}$ in $B_0,
 \dots  ,   B_3$ of degree $2 i$.

\begin{tiny}
\begin{equation}
\begin{split}
J_2 &    = -12B_0B_3+8B_0B_2+2B_1B_2+8B_1B_3-3B_1^2-3B_2^2 \\
J_4 & = -432B_0B_2B_1^2+608B_3B_1^2B_2-312B_0B_3B_1^2-1728B_0^2B_3B_2+960B_0B_2B_3^2-432B_1B_3B_2^2 -312B_0B_3B_2^2 -1728B_0B_3^2B_1 +608B_0B_1B_2^2\\
& +960B_3B_0^2B_1-2800B_0B_3B_1B_2+7056B_0^2B_3^2+528B_0B_2^3+528B_3B_1^3+256B_0^2B_2^2-122B_1^2B_2^2+256B_1^2B_3^2
-108B_1^3B_2-108B_1B_2^3+81B_1^4+81B_2^4  \\
J_6 &= -36B_3B_1^3B_2^2+118B_0^3B_3^2B_2-24B_0B_1^3B_2^2-8B_1^2B_2^2B_3^2-36B_0B_1^2B_2^3-24B_0^3B_2^3
-124B_0^3B_3^3-24B_1^3B_3^3-136B_3B_0^2B_2^3+8B_1B_2^4B_3+52B_1^3B_2B_3^2\\
& +36B_0B_1B_2^4-32B_3B_0^3B_2^2-32B_0B_1^2B_3^3-100B_3^2B_0^4-B_1^6-B_2^6
 -40B_0B_2^3B_3^2-10B_0^3B_3^2B_1+28B_0B_2^4B_3+8B_0B_1^4B_2-8B_0^2B_1^2B_2^2-38B_0^2B_2^2B_3^2\\
&
+140B_3B_0^3B_2B_1-100B_0^2B_3^4+36B_3B_1^4B_2-136B_0B_1^3B_3^2-38B_0^2B_1^2B_3^2-40B_3B_0^2B_1^3+52B_0^2B_1B_2^3-10B_0^2B_2B_3^3
 +118B_0^2B_1B_3^3-24B_3B_1^2B_2^3+28B_3B_0B_1^4\\
&-32B_0B_2^5-32B_3B_1^5+2B_1^5B_2+9B_1^4B_2^2-12B_1^3B_2^3 +9 B_1^2B_2^4+2B_1B_2^5+32B_0^2B_2^4+32B_1^4B_3^2+150B_0B_1^3B_2B_3-72B_0B_1^2B_2^2B_3-178B_0B_1^2B_2B_3^2\\
& +150B_0B_1B_2^3B_3-66B_0B_1B_2^2B_3^2-66B_3B_0^2B_1^2B_2-178B_3B_0^2B_1B_2^2+508B_0^2B_1B_2B_3^2+140B_0B_1B_2B_3^3\\
J_{10} & =
-37540800B_0^4B_3^5B_1-37540800B_0^5B_3^4B_2+148500B_0^3B_3^3B_2^4+148500B_0^3B_3^3B_1^4
-4028400B_0^4B_3^4B_1^2-860400B_0^2B_3^4B_1^4+5308200B_0^3B_3^4B_1^3\\
& +6696000B_0^5B_3^4B_1+6696000B_3^5B_0^4B_2+5308200B_3^3B_0^4B_2^3
-860400B_3^2B_0^4B_2^4-27000B_0^3B_3^4B_2^3-27000B_0^3B_3^2B_2^5
-25600B_0^2B_3^5B_1^3-100800B_0^3B_3^5B_1^2\\
& -44287200B_0^4B_3^4B_1B_2 -100800B_0^5B_3^3B_2^2-25600B_0^5B_3^2B_2^3-27000B_0^2B_3^3B_1^5
-27000B_0^4B_3^3B_1^3-4028400B_0^4B_3^4B_2^2-1854600B_0^3B_3^3B_1^2B_2^2\\
&
-543600B_0^3B_3^3B_1B_2^3+7719000B_0^3B_3^4B_1^2B_2+7719000B_0^4B_3^3B_1B_2^2-19800B_0^3B_3^2B_1^3B_2^2-543600B_0^3B_3^3B_1^3B_2
 +72600B_0^3B_3^2B_1^2B_2^3+142200B_0^3B_3^2B_1B_2^4\\
& -1225800B_0^3B_3^4B_1B_2^2+142200B_0^2B_3^3B_1^4B_2+351734400B_0^5B_3^5 +72600B_0^2B_3^3B_1^3B_2^2
-19800B_0^2B_3^3B_1^2B_2^3+146400B_0^4B_3^2B_1B_2^3+146400B_0^2B_3^4B_1^3B_2\\
& -18400B_0^2B_3^2B_1^3B_2^3+3600B_0^2B_3^4B_1^2B_2^2+3600B_0^4B_3^2B_1^2B_2^2+3600B_0^2B_3^2B_1^4B_2^2+3600B_0^2B_3^2B_1^2B_2^4-1225800B_0^4B_3^3B_1^2B_2\\
& -1080000B_3^6B_0^4-1080000B_0^6B_3^4+216000B_0^3B_3^5B_1B_2+216000B_0^5B_3^3B_2B_1\\
\end{split}
\end{equation}
\end{tiny}

Now the $T_i$ embed in $k(B_1, B_2, B_3)$ as follows
\begin{equation}
T_1  = \frac{J_4 (1,B_1, B_2, B_3)}{J_2 ^{2} (1,B_1, B_2, B_3)}, \, \, T_2 = \frac{J_6 (1,B_1, B_2, B_3)}{J_2 ^{3}
(1,B_1, B_2, B_3)}, \, \, T_3  = \frac {J_{10} (1,B_1, B_2, B_3)} {J_2 ^{5} (1,B_1, B_2, B_3)}
\end{equation}
Since $T_1$, $T_2$ and $T_3$ are independent variables over $k$ and $k(T_1,T_2,T_3,B_i) \subseteq k(B_1, B_2,
B_3)$ for $i$ = 1, 2 and 3, it follows $k(B_1, B_2, B_3)$ / $k(T_1, T_2, T_3)$ is a finite algebraic extension.
Also note $k(C_1, C_2, C_3)$ / $F$ is Galois  with group $S_6$ and $[k(B_1, B_2, B_3): F] = 120 $.

\medskip

{\it Proof of Theorem \ref{thm3}.} We know $k (T_1, T_2, T_3) \subseteq K \subseteq F$. The claim follows if $F =
k(T_1, T_2, T_3)$. Also $N:= [k(B_1,B_2,B_3): k(T_1,T_2,T_3)] $ is a multiple of $ [k(B_1,B_2,B_3) : F] = 120$.
Therefore,  if $N = 120 $, we are done. Let $\Omega$ be the algebraic closure of $ k(T_1, T_2, T_3)$. Then $N$ is
the number of embeddings $\alpha$ of $k(B_1, B_2, B_3)$ into $\Omega$ with $\alpha_{|k(T_1, T_2, T_3)} = id$.
Therefore,  the tuples $$(1, \alpha (B_1), \alpha (B_2), \alpha (B_3))$$ constitute $N$ distinct projective
solutions for the following system of homogeneous equations in $S_0, S_1, S_2, S_3$.
\begin{equation*}
\begin{split}
 T_1{J_2^{2} (S_0, S_1, S_2, S_3)} - {J_4 (S_0, S_1, S_2, S_3)} = 0\\
 T_2{J_2^{3} (S_0, S_1, S_2, S_3)} - {J_6 (S_0, S_1, S_2, S_3)} = 0  \\
 T_3{J_2^{5} (S_0, S_1, S_2, S_3)} - {J_{10}(S_0, S_1, S_2, S_3)} = 0\\
\end{split}
\end{equation*}
Besides these $N$ solutions there is the additional solution (0,0,0,1) by Lemma \ref{lem5}. Recall $J_{2i}$ are
homogeneous polynomials of degree $2i$. Therefore by Bezout's theorem,  $ N+1 \leq 4\cdot 6\cdot 10 = 240$. Hence,
$ N $ being a multiple of 120, must equal 120. This proves $ F = k(T_1,T_2,T_3) $.

\subsection{The Ring of Invariants of $GL_2(k)$  in $ k [A_0,  \dots  ,  A_6, I_{10}^{-1} ]$ }
\begin{thm}\label{thm4}
$\mathcal R = k\, [U_1, U_2,  \dots  ,  U_8]$ is the ring of $ \, GL_2(k)$-invariants \\
in  $k\, [A_0,  \dots  , A_6, I_{10}^{-1}]$.
\end{thm}

\begin{proof}
Let ${\mathcal R}_0 = k\, [A_0,  \dots  ,   A_6, I_{10}^{-1}]^{GL_2(k)}$. If $\frac{R}{S} \in \mathcal R_0$ with
$R$ and $S$ coprime polynomials, then $R$ and $S$ are homogeneous elements of same degree in $\mathcal R_6$ by
Remark \ref{rem4}. Since $S$ divides $I_{10} ^{e}$ (for some $e$) in $k\, [A_0,  \dots  ,   A_6]$, we have $SS^{'}
= I_{10}^{e}$ with $S^{'} \in \mathcal R_6$. Thus $\frac{R}{S} = \frac{RS^{'}}{I_{10}^{e}}= \frac{I}{I_{10}^{e}}$
with $I \in \mathcal R_6$.

We have $\mathcal R_{0} \subset K$. By Theorem \ref{thm3} we know $K$ is the field of fractions of $\mathcal R$.
By Lemma \ref{lem8} we know $\mathcal R$ is normal. Since  $\mathcal R \subseteq \mathcal R_{0} \subset K$, it
only remains to prove ${\mathcal R}_{0}$ is integral over $\mathcal R$. Let $u \in \mathcal R_0$. Then by the
preceding paragraph, $u= \frac{I}{I_{10}^{e}} $ with $I \in \mathcal R_6$. Thus $deg(I) = 10e$. Lemma \ref{lem6}
implies we have an equation
$$I^n + p_{n-1} I^{n-1} +  \dots  + p_0 = 0$$
where $p_i \in k\, [I_2,  \dots  ,   I_{10}]$. By dropping all terms of $degree \neq deg(I^n)$, we may assume
$p_i$ are homogeneous. Dividing by $I_{10}^{en}$ we have
$$u^n + \frac{p_{n-1}}{I_{10}^{e}} u^{n-1} +   \dots  + \frac{p_0} {I_{10}^{en}} = 0$$
where the coefficients lie in $\mathcal R$ by Lemma \ref{lem7}. This proves $\mathcal R_{0}$ is integral over
$\mathcal R$.
\end{proof}

\begin{cor}\label{cor1} (\textbf{Clebsch-Bolza-Igusa})
Two binary sextics $f$ and $g$ with $I_{10} \neq 0$ are $GL_2(k)$ conjugate if and only if there exists an $r \neq
0$ in $k$ such that for every $i$ = 1, 2, 3, 5 we have
\begin{equation}\label{eq9}
I_{2i}(f) = r^{2i} \, I_{2i}(g)
\end{equation}
\end{cor}

\begin{proof}
The only if part is clear.  Now assume Eq.~\eqref{eq9} holds. First note that we can assume the sextics to be of
the form $f(X,Y) = XY(X-Y)(X - a_1 Y)( X - a_2  Y)( X - a_3  Y) $ and $g (X,Y)  = XY(X-Y)(X - b_1  Y)(X - b_2 Y)(X
- b_3 Y)$ because every element in  $\cA$ is $GL_2(k)$ conjugate to  a  element in  $\cB$.  Now suppose  that they
are not  $GL_2(k)$ conjugate. Then  $ {\bf a} :=  ( a_1, a_2, a_3  ) $ and ${\bf  b} := ( b_1, b_2, b_3 )$ belong
to different $S_6$ orbits on $\mathcal C$ and these orbits are finite subsets of $ k^{3} $. Therefore there exists
a polynomial $ p(C_1, C_2, C_3) $ such that for all $ \tau \in S_6 $, we have $ p({\bf  a}^ \tau) =0 $ and $
p({\bf b}^ \tau) =  1 $. Consider the element  $s(C_1, C_2,  C_3) \in k\, [\mathcal  C] = k\, [C_1,  C_2, C_3,
\frac{1}{C_i}, \frac{1}{C_i - 1},  \frac{1}{C_i - C_j}]$ $(i,j = 1,2,3 \hspace{0.03in} and \hspace{0.03in} i \neq
j)$ given as
$$ s= \frac{1}{|S_6|} \sum_ {\substack {\tau \in S_6}} p({(C_1, C_2, C_3)}^ \tau).$$
Then $s$ takes the value 0 on ${\bf a}$ and 1 on ${\bf b}$. Clearly $s \in F = k(T_1, T_2, T_3) = k (U_1, U_2,
U_3) $. Let $q$ be a rational function in the  $S_6$ orbit of $p$. Then  from the explicit formulas for the $S_6$
action described earlier, we see  that the denominator of $q$ is a  product of the factors $C_i$, $C_i -1$,  $C_i
- C_j$ for all  $i$, $j$  =  1,2,3 and  $i  \neq j$.

The  sum  $ Q  = \sum_{\sigma \in S_3} q((C_1, C_2, C_3)^{\sigma})$ can be written as a quotient of  two symmetric
polynomials in $C_1$,  $C_2$, $C_3$. The denominator  is  a  product  of  factors mentioned  in  the  previous
paragraph and hence divides a  power of $J_{10} (1, B_1,$ $ B_2, B_3)$ in the ring  $k\, [B_1, B_2, B_3]$ ;  this
is because  $J_{10}(1, B_1, B_2, B_3)$ factors in $k\, [C_1, C_2, C_3]$ as
$$C_1^2C_2^2C_3^2(C_1-1)^2(C_2-1)^2(C_3-1)^2(C_1-C_2)^2(C_2-C_3)^2(C_3-C_1)^2.$$
Thus $Q \in k\, [B_1, B_2, B_3, J_{10}^{-1}]$ and hence $s \in k\, [B_1, B_2, B_3, J_{10}^{-1}]$.

Since $K = k(A_0,  \dots   ,   A_6)^{GL_2(k)} \cong k(C_1, C_2, C_3)^{S_6} = F$  by Theorem  \ref{thm3},  the
inverse image of $s$  in $K$  is a  rational function in $A_0$,  \dots  ,   $A_6$ which is defined at each point
of $\cB$ by the previous paragraph. Thus it is defined at  each point of $\cA$ because   it is
$GL_2(k)$-invariant. Therefore it lies   in $k\, [\cA]^{GL_2(k)}  =  k\, [A_0,    \dots   ,    A_6,
I_{10}^{-1}]^{GL_2(k)}  = \mathcal R$. But $\mathcal R = k\, [U_1,  \dots  ,   U_8]$ by Theorem 4. On the other
hand Eq.~\eqref{eq9}  implies that each $U_i$ takes the same value on $f$ and $g$. This implies $s$ takes  the
same value on ${\bf a}$ and ${\bf b}$, contradicting $s({\bf a}) = 0$ and $s({\bf b}) = 1$. This proves the claim.

\end{proof}

\section{Projective Invariance of unordered pairs of binary cubics}

\subsection{Null Cone of $ V_3 \bigoplus V_3$.}

In this  chapter $k$  is an algebraically  closed field  with $char(k) \neq 2, 3$. The Representation (see section
2.1) of $GL_2(k)$ in $V_3$ induces  a representation  of $GL_2(k)$  in $V_3  \bigoplus  V_3$. Let $\Gamma _0 $
$(\cong k^{*})$be the group of  maps $(f,g) \mapsto (cf, c^{-1}g)$, $c \in k^{*}$ on $V_3  \bigoplus V_3$. Let
$\Gamma $ be the semi-direct product  of $\Gamma _0$  and $< \nu  >$, where $\nu  : V_3 \bigoplus V_3  \to V_3
\bigoplus  V_3$ is $(f,g) \mapsto  (g,f)$. Then $\Gamma$ centralizes the $GL_2(k)$ action. Therefore we have an
action of $GL_2(k) \times \Gamma$ on $V_3 \bigoplus V_3$. The coordinate ring of $V_3 \bigoplus V_3 $ can be
identified with $k\, [{ A}_0,  \dots ,   { A}_3, {B}_0,   \dots ,   {  B}_3]$  where ${  A}_i$  and ${  B}_i$ are
coordinate functions  on  $V_3  \bigoplus  V_3$.  Let  $D_f$  and  $D_g$  be  the discriminants of the cubics
$f(X,Y) = A_0 X^3 + A_1 X^2Y + A_2 XY^2 + A_3  Y^3$ and  $g(X,Y) =  B_0 X^3  +  B_1 X^2Y  + B_2  XY^2 +B_3  Y^3$
respectively. Let $R$ be their resultant.

This  gives  the following  $SL_2(k)  \times \Gamma_0$-invariants  in  $k\, [A_0, \dots , A_3,  B_0, \dots , B_3]$
of   degree  4,6  and  8  respectively. $I = I_2(fg)$, $R$ and $D = D_f  D_g$. Further the skew symmetric form on
$V_3$ yields a $SL_2(k) \times \Gamma _0$-invariant  $H$ of degree 2. These are listed below.

\begin{small}
\begin{equation}
\begin{split}
H &= 3 {A}_0{{ B}_3} - { A}_1{{ B}_2} + { A}_2{{ B}_1} - 3 { A}_3{{ B}_0}\\
I &= 228{{ A}_0}{{ B}_0}{{ A}_3}{{ B}_3}-52{{ A}_1}{{ B}_0}{{ A}_3}{{ B}_2}-24{{ A}_1}{{ B}_0}{{ A}_2}{{ B}_3}-24{{ A}_0}{{ B}_1}{{ A}_3}{{ B}_2}-52{{ A}_0}{{ B}_1}{{ A}_2}{{ B}_3}\\
& +4{{ A}_2}{{ B}_0}{{ A}_3}{{ B}_1}+16{{ A}_2}^2{{ B}_0}{{ B}_2}+16{{ A}_1}{{ B}_1}^2{{ A}_3}+4{{ A}_1}{{ B}_1}{{ A}_2}{{ B}_2}+16{{ A}_1}^2{{ B}_1}{{ B}_3}\\
& +16{{ A}_0}{{ B}_2}^2{{ A}_2}+4{{ A}_0}{{ B}_2}{{ A}_1}{{ B}_3}-6{{ A}_3}^2{{ B}_0}^2-6{{ A}_2}^2{{ B}_1}^2-6{{ A}_1}^2{{ B}_2}^2-6{{ A}_0}^2{{ B}_3}^2 \\
R & = 3B_0^2A_0B_3A_3^2-B_0^3A_3^3+2B_0^2A_3^2B_2A_1-B_2^2B_0A_1^2A_3-A_0^2B_2^3A_3\\
& + B_0^2A_2B_1A_3^2-B_0^2A_2^2B_2A_3-B_1^2B_0A_1A_3^2+A_0B_1^3A_3^2-3B_0A_0^2B_3^2A_3\\
& -B_0A_1^3B_3^2+A_0^3B_3^3+B_0^2A_2^3B_3-B_0A_0B_3B_2A_1A_3+3A_0^2B_3B_2B_1A_3+B_0A_3B_3A_0B_1A_2\\
& +3B_0A_0B_3^2A_1A_2-2A_0^2B_3^2B_1A_2-3B_0A_3^2B_2A_0B_1-3B_0^2A_3B_3A_1A_2-B_2A_1A_0^2B_3^2\\
& +B_2^2A_1A_0B_1A_3+B_2B_0A_1^2B_3A_2-B_2A_1B_3A_0B_1A_2+A_0^2B_2^2A_2B_3+2B_0A_0B_2^2A_2A_3\\
& -2B_0A_0B_2B_3A_2^2-2B_1^2A_1A_3A_0B_3+B_1A_1^2B_3^2A_0+2B_1B_0A_1^2B_3A_3+B_1B_0A_1B_2A_2A_3\\
& -B_1 B_0 A_1 B_3 A_2^2-A_0 B_1^2 B_2 A_2 A_3+A_0 B_1^2 B_3 A_2^2\\
D & = (-27A_0^2A_3^2+18A_0A_3A_2A_1+A_1^2A_2^2-4A_1^3A_3-4A_2^3A_0)(-27B_0^2B_3^2+18B_0B_3B_2B_1\\
&+B_1^2B_2^2-4B_1^3B_3-4B_2^3B_0)\\
\end{split}
\end{equation}
\end{small}

Note that $R$ and  $ H$ change by a sign if  the cubics are switched (i.e.,   they  are   not  $\nu$-invariant)
but  $I$   and   $D$  are $\nu$-invariant.

\begin{defn}
Let ${\mathcal R}_{(3,3)}$ denote  the ring of $SL_2(k) \times \Gamma_0$-invariants in
$$k\, [A_0,  \dots  ,   A_3,  B_0,  \dots  ,   B_3].$$
The  null cone  $\mathcal  N_{(3,3)}$  is the  common  zero  set  of all  homogeneous  elements of positive degree
in $\mathcal R_{(3,3)}$
\end{defn}

\begin{lem}\label{lem10}
(i).  Let $f,g \in V_3$. Then $fg=0$ or has a root of multiplicity at least 4 if and only if $D$, $R$, $ H$ and
$I$ vanish simultaneously on the pair $(f,g)$.

(ii). The null cone $\mathcal N_{(3,3)}$ is the common zero set of $D$, $R$, $I$ and $H$.

(iii).  ${\mathcal R}_{(3,3)}$ is finitely generated as a module over $k\, [D, R, H, I]$.
\end{lem}

\begin{proof}
(i) If $f g$ has a root of multiplicity four then $f$ and $g$ must have a common root. Therefore $R = 0$. Moreover
this common root must be of multiplicity at least 2 in either $f$ or $g$ and hence $D = 0$. Also from Lemma
\ref{lem5} we know $I = 0$. One also checks that $H = 0$. Conversely, let $D = R = H = I = 0$. Recall $D = {D_f}
D_g $ where $D_f$ and $D_g$ are discriminants of $f$ and $g$ respectively. We may assume $f \neq 0 \neq g$. Say
$D_f = 0$. Then we may assume $f = X^3$  or $X^2Y$.

\smallskip

{\it Case(1):} $f = X^3$. Since $R = 0$, we get $X$ divides $g$ and hence $X^4$ divides $fg$.

\smallskip

{\it Case(2):} $f = X^2 Y$. Since $R = 0$, either $X$ or $Y$ divides $g$. Thus $g = X(aX^2 + bXY + cY^2)$ or $g =
Y(aX^2 + bXY + cY^2)$.

(2a): Let $ g = X(aX^2 + bXY + cY^2)$. Then $H = -c$. Therefore $H = 0$ implies $c = 0 $ and hence $X^4$ divides
$fg$.

(2b): Let $g = Y(aX^2 + bXY + cY^2)$. Then $H = -b$ and $I = 16ac$. Therefore $H = I = 0$ implies $a = b = 0$ or $
b = c = 0$ and hence $Y^4$ divides $fg$ or $X^4$ divides $fg$.

\smallskip

(ii) : Suppose $I \in \mathcal R_{(3,3)}$ is homogeneous of degree $s > 0$. We know $I(f,g) = I(cf, c^{-1}g)$ for
every $c \in k^{*}$. Then  $I(f,0) = I(cf,0)$  for every $c \in k^{*}$, so  $I(f,0)$ viewed as a  polynomial  in
$A_0$,   \dots  ,    $A_3$ is  constant and  hence is  0 (by  taking $f = 0$). Rest is as in Lemma \ref{lem5}.

\smallskip

(iii) : The claim follows because the analogue of Theorem \ref{thm2} holds here (with the same proof).

\end{proof}


\begin{rem}\label{rem5}
(a) Since $V_3 \bigoplus V_3$ / $SL_2(k) \times \Gamma _0$ is a 4 dimensional variety, Lemma \ref{lem10} (iii)
implies $D$, $R$, $H$ and $I$ are algebraically independent over $k$.

(b) The  quotient of two homogeneous elements  in $\mathcal R_{(3,3)}$ of  the same  degree is  $GL_2(k) \times
\Gamma$-invariant if and only if   it is $\nu$-invariant.  In particular  the following  elements  are $GL_2(k)
\times \Gamma$-invariants.
$$ R_1 := \frac{{H}^2}{I}, \quad R_2 := \frac{{H}^3}{R}, \quad R_3 := \frac{{H}^4}{D} $$

(c)  Assertion  (a)  implies  $R_1$, $R_2$,  $R_3$  are  algebraically independent over $k$. The proof of this
fact is similar to   Remark \ref{rem2} (c) (for $\frac{1}{R_1}$, $\frac{1}{R_2}$, $\frac{1}{R_3}$).
\end{rem}
Further   define the following
\begin{equation}
\begin{split}
V_1 &  := \frac{I H}{R} = \frac{R_2}{R_1}, \quad V_2:= \frac{H ^{3}}{R} = R_2, \quad V_3:= \frac{{H}^4}{D} = R_3,\\
V_4 & := \frac{I ^{2}}{D} = \frac{R_3}{R_1 ^{2}}, \quad V_5:= \frac{I ^{3}}{R^2} = \frac{R_2 ^{2}}{R_1 ^{3}},
  \quad  V_6:= \frac{I H ^{2}}{D} = \frac{R_3}{R_1}\\
\end{split}
\end{equation}

\begin{rem}\label{rem6}
The definitions of $V_1$, $V_2$, $V_3$ imply $k ( R_1, R_2, R_3) = k ( V_1, V_2, V_3)$. Therefore $V_1$, $V_2$ and
$V_3$ are also algebraically independent over $k$.
\end{rem}

\begin{lem}\label{lem11}
Let $a$, $b$, $c$ and $d$ be non-negative integers such that $a+2b = 3c+4d$. Then ${\bf m} = \frac{H ^{a} I
^{b}}{R^c D^d} \in k\, [V_1, V_2,  \dots  ,  V_6]$.
\end{lem}

\begin{proof}
Extracting powers of $V_2$ and $V_3$ we may assume $a \leq 3$ and extracting powers of $V_4$ and $V_5$ we may
assume $b \leq 1$. This gives six possibilities for the pair $(a,b)$ and this leads to $V_1,  \dots  , V_6$.
\end{proof}

\begin{lem}\label{lem12}
The ring $ \mathcal S = k\, [V_1, V_2, V_3, V_4, V_5, V_6]$ is normal.
\end{lem}

\begin{proof}
Suppose an element $U$ in the field of fractions of $\mathcal S$ is integral over $\mathcal S$. Then we have an
equation.
\begin{equation}
U^n + p_{n-1}(V_1,  \dots  ,   V_6)U^{n-1} +  \dots  + p_0( V_1,  \dots  ,   V_6) = 0
\end{equation}
where $p_i$ are polynomials in 6 variables over $k$. Let $e$ be a positive integer such that $(R D)^{e}p_i \in k\,
[H,I, R, D]$. Then multiplying the above equation  by $(RD)^{en}$, we see that $(R D)^e U$ is integral over $k\,
[H, I, R, D]$. By  Remark  \ref{rem5}, (a) we know that $k\, [H, I, R, D]$ is a polynomial ring. Also the field of
fractions of $\mathcal S$ is contained in $k(H, I, R, D)$. Therefore $(RD)^e U \in k\, [H, I, R, D]$. Lemma
\ref{lem11} implies $U \in \mathcal S$.
\end{proof}

\subsection{The Field of Invariants of $GL_2(k) \times \Gamma$ in $k(A_0, \dots , A_3, B_0, \dots , B_3)$.}

\begin{thm}\label{thm5}
The  field  $L$  of  $GL_2(k)  \times  \Gamma$-invariants  in  $k(A_0,  \dots  , A_3, B_0,  \dots , B_3)$ is a
rational function field ,   namely $ L = k\, ( R_1, R_2, R_3) = k\, (V_1, V_2, V_3)$.
\end{thm}

By   Remark \ref{rem6} we only  have to show  $L = k(R_1, R_2,  R_3)$. The rest of this section occupies the
proof.

\begin{rem}\label{rem7}
If $\frac {T}{S} \in  L$ with $T$ and $S$  coprime polynomials, then it follows as in   Remark \ref{rem4} that $T
= c_g T^{g}$, $S = c_g S^{g}$ for every $g \in GL_2(k) \times \Gamma$ and  $c_g = 1$ for $ g \in SL_2(k) \times
\Gamma _0$,  the   commutator  subgroup  of  $GL_2(k)  \times \Gamma$. Thus $T$, $S \in \mathcal R_{(3,3)}$.
Further $T$ and $S$ are homogeneous of the same degree.
\end{rem}

\noindent We introduce the following notations.
\begin{equation}
\begin{split}
{\bar {\mathcal A}}  & := \{ (f,g) \in V_3 \bigoplus V_3: \, R(f,g) \cdot D(f,g) \neq 0 \}\\
{\bar {\mathcal B}}  &:= \{ (XY(X-Y), f_3): f_3 = X^3+b_1 X^2 Y+b_2 X Y^2+b_3  Y^3 \\
 & = (X-c_1 Y)(X-c_2 Y)(X-c_3 Y), \, (0, 1,\infty, c_1, c_2, c_3)\in {\mathcal C} \}
\end{split}
\end{equation}
Let $B_i$ be functions on $\bar {\cB} $ mapping $(\, XY(X-Y), \, X^3 Y + { b}_1 X^2 Y + { b}_2 XY^2 +{ b}_3 Y^3 )
\mapsto { b}_i $. Then $k (\bar {\cB}) = k \, ({ B}_1, { B}_2, { B}_3) \subset k(\mathcal C)$ . Let $M$ denote the
fixed field of the action of $(S_3 \times S_3) \rtimes \Z _2 = S_3 \wr \Z _2 < S_6$ on $k\, ( C_1, C_2, C_3)$.
Here $S_3 \wr \Z_2$ denotes the wreath product.

\begin{figure}[ht!]
$$
\xymatrix{ &&  k(C_1, C_2, C_3)\ar@{-}[d]^{6} &&\\&&  k(B_1, B_2, B_3) \ar@{-}[d]^{12}&& \\ &&M}$$
\end{figure}

\begin{lem}\label{lem13}
The inclusion $\bar \cB \subset V_3 \bigoplus V_3$ yields an embedding
$$L \subseteq M \subset k(B_1, B_2, B_3).$$
\end{lem}

\begin{proof}
Note that any $(f,g) \in V_3 \bigoplus V_3 $ with $R(f,g) \cdot D(f,g) \neq 0$ is $GL_2(k) \times
\Gamma$-conjugate to an element in $\bar \cB$. Indeed, using $SL_2(k)$ we can move the roots of $f$ to $(1,0)$,
$(0,1)$ and $(1,1)$. Then $f$ becomes a scalar multiple of $XY (X-Y)$. Further we can replace $f$ and $g$ by
scalar multiples because given $c \in k^{*}$, there are elements $\gamma_1$, $\gamma_2 \in GL_2(k) \times \Gamma$
such that $(f,g)^{\gamma_1} = (c f, g)$ and $(f,g)^{\gamma_2} = (f, \, c g)$.

If $\frac{T}{S} \in L$ with $T$ and $S$ coprime polynomials, then $S$ does not vanish on $\bar \cB$ by the
previous paragraph. Therefore the restriction map $ L \to k(\bar \cB) \subset k(\mathcal C)$ is well defined. Let
$I \in L$ and $\bar I$ its image in $k(\mathcal C)$. Denote $p = (0, 1, \infty, c_1, c_2, c_3) \in \mathcal C$ by
$(p_1, p_2,  \dots  ,  p_6)$. For $\tau \in S_3 \wr \Z _2 < S_6$, we have
\begin{small}
\begin{equation*}
\begin{split}
 \bar I (p^{\tau}) = & \bar I (g(p_{\tau(1)}), \dots  ,  g(p_{\tau(6)})) = I ( (X-g(p_{\tau(1)})Y)
(X-g(p_{\tau(2)})Y) (X-g(p_{\tau(3)})Y),\\
 & (X-g(p_{\tau(4)})Y) (X-g(p_{\tau(5)})Y) (X-g(p_{\tau(6)})Y) ) \\
  = & I( (X-p_{\tau(1)}Y)(X-p_{\tau(2)}Y)(X-p_{\tau(3)}Y), (X-p_{\tau(4)}Y)(X-p_{\tau(5)}Y)(X-p_{\tau(6)}Y)
\end{split}
\end{equation*}
\end{small}
for some $g \in GL_2(k)$. But $\{ \tau(1), \tau(2), \tau(3) \}$ equals $\{1, 2, 3 \}$ or $\{ 4, 5, 6 \}$ and $I$
is symmetric in $f$ and $g$, it follows $\bar I(p^{\tau}) = \bar I(p)$. Thus $\bar I \in M$.

\end{proof}

The evaluation of $H$, $I$, $R$ and $D$ \ \  on $( XY(X-Y), \, {b_0}X^3 +{b_1}X^2 Y + {b_2}XY^2 + b_3 Y^3 )$ gives
the following homogeneous polynomials of degree 1, 2, 3 and 4 respectively.

\begin{small}
\begin{equation}
\begin{split}
& \w{H} \, (B_0, B_1, B_2, B_3)  = - (B_1+B_2) \\
& \w{I} \, (B_0,B_1,B_2,B_3) = 24{B_3}{B_0}+16{B_2}{B_0}-4{B_1}{B_2}+16{B_1}{B_3}-6{B_1}^2-6{B_2}^2 \\
& \w{R} \, (B_0, B_1, B_2, B_3) = {B_0}{B_3}(B_0 + B_1 + B_2 + B_3)\\
& \w{D} \, (B_0,B_1,B_2,B_3) = -4{B_0}{B_2}^3+{B_1}^2{B_2}^2+18{B_0}{B_1}{B_2}{B_3}-4{B_1}^3{B_3}-27{B_0}^2{B_3}^2\\
\end{split}
\end{equation}
\end{small}

\noindent  Thus the elements $R_1$, $R_2$ and $R_3$ of $L$ embed in $k\, (B_1, B_2, B_3)$ as follows

\begin{small}
$$ R_1 = \frac{{\w{H} ^2} (1,B_1,B_2,B_3)}{\w{I}(1,B_1,B_2,B_3)}, \quad
R_2 = \frac{{\w{H} ^3} (1,B_1,B_2,B_3)}{\w{R}(1,B_1,B_2,B_3)}, \quad R_3 = \frac{{\w{H} ^4}
(1,B_1,B_2,B_3)}{\w{D}(1,B_1,B_2,B_3)}. $$
\end{small}

\medskip

\noindent  {\it Proof of Theorem \ref{thm5} }. We know $k(R_1, R_2, R_3) \subseteq L \subseteq M$. The theorem
follows if $M = k(R_1, R_2, R_3) $. Furthermore $$m := [ k (B_1, B_2, B_3) : k (R_1, R_2, R_3)]$$ is a multiple of
$[ k(B_1, B_2, B_3) : M] = 12$. Therefore the claim follows if $m = 12$. Let $\Lambda$ be the algebraic closure of
$k\, (R_1, R_2, R_3)$. Then $m$ is the number of embeddings $\beta$ of $k(B_1, B_2, B_3)$ into $\Lambda$ with
$\beta_ {| k(R_1, R_2, R_3)} = id $. Therefore the tuples $(1, \beta(B_1), \beta(B_2), \beta(B_3))$ constitute $m$
distinct projective solutions for the following system of homogeneous equations in $S_0$, $S_1$, $S_2$ and $S_3$.
\begin{equation}
\begin{split}
  {\w{H}}^2 (S_0,  \dots  ,   S_3) - R_1 \w{I} (S_0,  \dots  ,   S_3) = 0\\
  {\w{H}}^3 (S_0,  \dots  ,   S_3) - R_2 \w{R} (S_0,  \dots  ,   S_3) = 0\\
  {\w{H}}^4 (S_0,  \dots  ,   S_3) - R_3 \w{D} (S_0,  \dots  ,   S_3) = 0\\
\end{split}
\end{equation}
Besides these $m$ solutions there is the additional solution $(0,0,0,1)$. Therefore by $Bezout's$ theorem, $m+1
\leq 2 \cdot 3 \cdot 4 = 24$. Hence, $m$ being a multiple of 12, must equal 12. This proves $L = k(R_1, R_2,
R_3)$.

\subsection{ The Ring of $GL_2(k) \times \Gamma$-invariants
in $k\, [A_0,  \dots  ,  A_3,B_0,  \dots  ,  B_3,R^{-1},D^{-1}]$}

In this section we prove the following:

\begin{thm}\label{thm6}
$\mathcal S = k\, [V_1, V_2,  \dots  ,   V_6]$ is the ring of $GL_2(k) \times \Gamma$-invariants in

 $k\, [A_0, \dots  , A_3, B_0,  \dots  , B_3, R^{-1}, D^{-1}]$.
\end{thm}

\begin{proof}
Let ${\mathcal  S}_{0} =  k\, [A_0,\dots , A_3, B_0, \dots  , B_3, R^{-1}, D^{-1}]^{GL_2(k) \times \Gamma}$. If
$\frac{T}{S}  \in \mathcal S_0$ with  $T$  and  $S$  coprime   polynomials,  then  $T$  and  $S$  are homogeneous
elements  of $\mathcal R_{(3,3)}$  of the same  degree by   Remark \ref{rem7}.  Since  $S$  divides $(RD)^{e}$
(for some $e$)  in $k\, [A_0,  \dots  ,  A_3, B_0,  \dots  ,    B_3]$, we have $SS^{'} = (RD)^{e}$ with $S^{'} \in
\mathcal R_{(3,3)}$.    Thus
$$\frac{T}{S}    = \frac{T S^{'}} {(RD)^{e}}  = \frac{I}{(RD)^{e}}$$
with $I  \in \mathcal R_{(3,3)}$.

We have  $\mathcal S_{0} \subset L$. Further by Theorem \ref{thm5} we  know $L$ is the field of fractions  of
$\mathcal S$. By Lemma \ref{lem12} we know $\mathcal  S$ is normal. Since $\mathcal S \subseteq \mathcal S_{0}
\subset L$, it only remains to prove $\mathcal S_{0}$ is integral over $\mathcal S$. Let $u \in \mathcal S_0$.
Then by the previous paragraph $ u  = \frac{I}{(RD)^{e}}$  with $I  \in \mathcal R_{(3,3)}$. Thus, \ $deg(I) =
14e$. \ \ Lemma \ref{lem10}  (iii) implies
\begin{equation}
I^n + p_{n-1}I^{n-1} +  \dots +p_0 = 0
\end{equation}
where $p_i \in  k\, [H, I, R, D]$. By dropping all  terms of $degree \neq deg(I^n)$,  we   may  assume  $p_i$  are
homogeneous.   Dividing  by $(RD)^{en}$ we have
$$u^n + \frac{p_{n-1}}{(RD)^{e}}  u^{n-1} +  \dots  + \frac{p_0}{(RD)^{ne}} = 0$$
where the coefficients  lie in $\mathcal S$, by  Lemma \ref{lem11}. This proves $\mathcal S_{0}$ is integral over
$\mathcal S$.

\end{proof}

\begin{cor}\label{cor2}
Suppose $\{P, Q\}$  and $\{ P^{'}, Q^{'} \} $  are two unordered pairs of  disjoint  3-sets in  $\mathbb  P^{1}$.
They  are conjugate  under $PGL_2(k)$ if  and only if $V_1$,  \dots   ,   $V_6$ take the  same value on the two
pairs.
\end{cor}

\begin{cor}\label{cor3} Two pairs $(f_1, f_2)$, $(g_1, g_2) \in V_3 \bigoplus V_3$ with
$R(f_1, f_2)\cdot D(f_1, f_2) \neq 0$ and $R(g_1, g_2) \cdot D(g_1, g_2) \neq 0$ are $GL_2(k) \times
\Gamma$-conjugate if and only if there exists an $r \neq 0$ in $k$ such that
\begin{equation}\label{eq_4.2}
\begin{split}
 H(f_1, f_2) = r^2 H(g_1, g_2)\\
I(f_1, f_2) = r^4 I(g_1, g_2)\\
 R(f_1, f_2) = r^6 R(g_1, g_2)\\
 D(f_1, f_2) = r^8 D(g_1, g_2) \\
\end{split}
\end{equation}

\end{cor}

\begin{proof}
The only if part is clear.  Now assume Eq.~\eqref{eq_4.2} holds. We can assume 
\[ f_1 = g_1  = XY(X-Y),\]
$f_2$ and $g_2$ equals $(X-\alpha_1 Y)(X-\alpha_2 Y)(X-\alpha_3  Y)$ and  $(X-\beta_1 Y)(X-\beta_2  Y)(X-\beta_3 Y)$ respectively. This is  because  every element  in  $\cA$ is  $GL_2(k) \times \Gamma$-conjugate to an element  in $\cB$. Suppose they are not $GL_2(k)  \times \Gamma$-conjugate. Then  ${\bf \alpha}  := (\alpha_1, \alpha_2, \alpha_3)$ and ${\bf \beta} := (\beta_1, \beta_2, \beta_3)$ belong to different $S_3 \wr \Z  _2 $ orbits on $\mathcal C$ and these orbits  are  finite  subsets  of  $k ^3$.  Therefore  there  exists  a polynomial \, \,  $p (C_1, C_2, C_3)$ such that for all $  \tau \in S_3 \wr \Z _2  $, we have $  p({\bf \alpha}^ \tau)  =0 $ and $ p({\bf \beta}^ \tau) = 1 $. Consider the element $t \in k\, [\mathcal C] $ given as
\[ t= \frac{1}{|(S_3 \wr \Z _2|}  \sum_ {\substack {\tau \in S_3 \wr \Z _2}} p({(C_1, C_2, C_3)}^ \tau)\]
Clearly $t \in M$. As in  the proof of Corollary \ref{cor1}, we have
\[ t \in k\, [B_1, B_2, B_3, J_{10}^{-1}] = k\, [B_1, B_2, B_3, R^{-1}, D^{-1}].\]

Since
\[ L = k(A_0,  \dots  ,   A_3, B_0,  \dots  ,   B_3)^{GL_2(k) \times \Gamma} \cong k(C_1, C_2, C_3)^{S_3 \wr\Z _2} = M\]
by Theorem \ref{thm5}, the inverse image of $t$ in $L$ is a rational function in $A_0,  \dots , A_3$, $B_0,  \dots  ,   B_3$ which is defined at each point of $\bar \cB$ by the previous paragraph. Thus it is defined at each point of $\bar \cA$ because it is a $GL_2(k) \times \Gamma$-invariant. Therefore it lies in
\[ k\, [\bar \cA]^{GL_2(k) \times \Gamma} = k\, [A_0, \dots , A_3, B_0,  \dots , B_3, R^{-1}, D^{-1}]^{GL_2(k) \times \Gamma} = \mathcal S.\]
But $S = k\, [V_1,  \dots  ,   V_6]$ by Theorem~\ref{thm6}. On the other hand Eq.~\eqref{eq_4.2} implies each $V_i$ takes the same value on $(f_1, f_2)$ and $(g_1, g_2)$. This implies $t$ takes the same value on ${\bf \alpha}$ and ${\bf \beta}$, 
contradicting $t({\bf \alpha}) = 0$ and $t({\bf \beta}) = 1$. This proves the claim.

\end{proof}

\end{document}